\documentclass[10pt]{article}

\usepackage{amsmath}
\usepackage{amssymb}
\usepackage{amsthm}

\newcommand{\R}{\mathbb{R}}
\newcommand{\C}{\mathbb{C}}
\newcommand{\Z}{\mathbb{Z}}
\newcommand{\N}{\mathbb{N}}
\newcommand{\cc}[1]{\overline{#1}}
\newcommand{\bs}{\boldsymbol}

\newcommand{\Span}{\mathop{\mathrm{span}}}
\newcommand{\te}{\tilde}
\newcommand{\ind}[2]{\mathstrut^{#1}\!{#2}}
\newcommand{\cl}[1]{{\cal #1}}
\newcommand{\set}[2]{\left\{#1:#2\right\}}
\newcommand{\cA}{{\cal A}}
\newcommand{\cS}{{\cal S}}
\newcommand{\cD}{{\cal D}}
\newcommand{\cV}{{\cal V}}
\newcommand{\cP}{{\cal P}}
\newcommand{\bD}{{\bs D}}
\newcommand{\bP}{{\bs P}}

\newcommand{\lxl}{\mathrel{<_{\mathrm{lex}}}}
\newcommand{\ssc}[2]{\mathstrut^{#1}\!#2}

\newtheorem{theorem}{Theorem}[section]
\newtheorem{lemma}[theorem]{Lemma}
\newtheorem{corollary}[theorem]{Corollary}
\newtheorem{statement}[theorem]{Statement}

\newtheorem{remark}{Remark}[section]
\newtheorem{definition}{Definition}[section]

\numberwithin{equation}{section}

\begin{document}

\title{Polynomial spaces reproduced by elliptic scaling functions\thanks{Submitted to Constr.\ Approx.}}

\author{Victor G.~Zakharov\\[0.5ex]
\small Institute of Continuum Mechanics of Russian Academy of Sciences,\\
\small              Perm, 614013, Russia\\
\small  E-mail: \tt    victor@icmm.ru
}

\maketitle{}

\paragraph{Abstract}
The Strang--Fix conditions are necessary and sufficient to reproduce spaces of algebraic
polynomials up to some degree by integer shifts of compactly supported functions.
W.~Dahmen and Ch.~Micchelli (Linear Algebra Appl.\ 52/3:217--234,\,1983) introduced a
generalization of the Strang--Fix conditions to affinely invariant
subspaces of higher degree polynomials.
C.~de~Boor (Constr.\ Approx.\
3:199–-208,\,1987)
raised a question on the necessity of scale-invariance of polynomial space for an arbitrary function; and he
omitted the scale-invariance restriction on the space.
In the paper, we present a matrix approach to determine (not necessarily scale-invariant)
polynomial space contained
in the span of integer shifts of a compactly supported function.
Also, in the paper, we consider scaling functions that we call
the elliptic scaling functions (Int.\ J.\ Wavelets Multiresolut.\
Inf.\ Process.\,To appear); and, using the matrix approach, we
prove that the elliptic scaling functions satisfy the generalized
Strang--Fix conditions and reproduce only affinely invariant
polynomial spaces. Namely we prove that any algebraic polynomial
contained in the span of integer shifts of a compactly supported
elliptic scaling function belongs to the null-space
of a homogeneous elliptic differential operator.
However, in the paper, we present
nonstationary elliptic scaling functions such that
the scaling functions
reproduce not scale-invariant (only shift-invariant) polynomial spaces.

\medskip



\section{Introduction}

In the paper~\cite{DahmenMicchelli}, W.\;Dahmen and Ch.\;Micchelli introduced a space of polynomials
\begin{equation}\label{SDefinition}
  {\cV}:=\set{P}{P\in\Pi,\,\left(P(D)\hat f\right)(2\pi k)=0,\ k\in\Z^d\setminus\{0\}},
\end{equation}
where $\Pi$ is the space of {\em all} polynomials on $\R^d$,
$P(D)$ is the differential operator induced by $P$, and $f$ is a
compactly supported function (that, for example, belongs to the
space of tempered distributions). And, see Proposition~2.1 in the
paper~\cite{DahmenMicchelli}, it was proved that if there
exists an affinely, i.\,e., shift- and scale-, invariant subspace
$\cV_{\rm aff}\subseteq\cV$, where $\cV$ is given
by~\eqref{SDefinition}, and $\hat f(0)\ne0$; then the span of
integer shifts of $f$ contains the space $\cV_{\rm aff}$.

The conditions on a function $f$:
\begin{equation}\label{Strang-FixConditions}
  \left(P(D)\hat f\right)(2\pi k)=0,\ k\in\Z^d\setminus\{0\},\ \forall P\in\cV\subset\Pi
\end{equation}
can be considered as a generalization of the Strang--Fix conditions~\cite{FixStr,StrFix}.
Unlike the classical Strang--Fix conditions, the derivatives of the Fourier
transform of the function $f$ do not necessarily vanish to satisfy~\eqref{Strang-FixConditions}.
The well-known functions that satisfy conditions~\eqref{Strang-FixConditions} in the nontrivial case,
i.\,e., if not all the derivatives up to some order vanish,
are the box-splines.

In the paper~\cite{R.-Q.Jia}, R.-Q.~Jia proved that if a multivariate scaling function that
satisfies a refinement relation with an {\em isotropic} dilation matrix belongs
to the Sobolev space $W_1^{k}(\R^d)$;
then the scaling function satisfies the (classical) Strang--Fix conditions
of order $k$, i.\,e,
the scaling function reproduces {\em all} the polynomials up to degree $k$.
In the present paper, we consider
the so-called {\em elliptic scaling functions}~\cite{ZakhConstrApprox}.
The elliptic scaling functions satisfy refinement relations with real isotropic dilation matrices.
In the paper, we prove that any
real isotropic matrix is similar to an orthogonal matrix and the similarity transformation matrix defines a
positive definite quadratic form.
The quadratic form determines homogeneous elliptic differential operators
and the form (and operators) is invariant under
coordinate transformation by the dilation matrix.
The elliptic scaling functions
also satisfy nontrivial conditions~\eqref{Strang-FixConditions}.
In fact, the algebraic polynomials reproduced by compactly supported elliptic scaling functions
belong to the null-spaces
of homogeneous elliptic differential operators.
Note that many essential properties of the elliptic scaling functions are
similar to the properties of the (univariate and multivariate) B-splines.
We refer the reader to~\cite{ZakhConstrApprox} for details.

The affine invariance of the polynomial spaces
seems rather strong restriction and,
in the paper~\cite[Proposition~2.2]{deBoor} of
C.~de~Boor, a generalization to not scale-invariant (only shift-invariant) polynomial
spaces was considered.
Namely it was proved that if a compactly supported
function $f$ satisfies conditions~\eqref{Strang-FixConditions} (and $\hat f(0)\ne0$);
then
the span of integer shifts of the function $f$ contains the {\em largest shift-invariant subspace} of the space $\cV$.
However, in the paper~\cite{deBoorHollig},
it was shown that the box-splines reproduce only affinely invariant polynomial spaces.
In the case of the elliptic scaling functions, we also have affinely invariant spaces only.
Nevertheless a nonstationary generalization,
i.\,e., if the scaling functions (and the corresponding masks) do not coincide
for different scales,
of the elliptic scaling functions allows satisfying
properties~\eqref{Strang-FixConditions} in the not scale-invariant case. And
we
construct nonstationary scaling functions that
reproduce not affinely invariant polynomial spaces.

Note that, generally, the determination of a shift-invariant subspace of space~\eqref{SDefinition} is
a nontrivial problem. In the paper, we present an approach to determine
the largest shift-invariant
subspace of the space $\cV$. In fact,
we introduce a matrix of (non-zero) derivatives of the function $\hat f$ at the points $2\pi\Z^d\setminus\{0\}$;
and the null-space of the matrix
defines completely the largest shift-invariant subspace of the space $\cV$.
(In the framework of this approach, the affinely invariant subspaces also can be obtained.)

The paper is organized as follows.
Basing on the papers of C.~de~ Boor~\cite{deBoor} and
W.~Dahmen \& Ch.~Micchelli~\cite{DahmenMicchelli},
Section~\ref{Sect:WeakenedStrang--FixConditions}
is devoted to
a generalization of the Strang--Fix conditions.
In particular, in the section, we introduce notations and definitions.
In Section~\ref{Sect:EllipticScalingFunctions}, we consider the so-called elliptic scaling functions.
The section is based on the
paper~\cite{ZakhConstrApprox} and contains a definition of the elliptic scaling functions
and presents some properties of the elliptic scaling functions.
In particular,
the positive definite quadratic forms that correspond
to the dilation matrices and
determine the homogeneous elliptic differential
operators are defined.
Section~\ref{Sect:ReproductionOfPolynomials}
is devoted to the polynomial spaces
reproduced by compactly supported elliptic scaling functions.
In Subsection~\ref{Subsect:NotAffinelyInvariantCase}, we consider a nonstationary generalization of the
elliptic scaling functions and prove that the polynomial spaces contained in the spans of integer
shifts of the nonstationary elliptic scaling functions can be not scale-invariant.

\subsection{Notations}

Here we introduce some general notation.

A {\em multi-index} $\alpha$ is a $d$-tuple
$(\alpha_1,\dots,\alpha_d)$ with its components being nonnegative
integers, i.\,e., $\alpha\in\Z^d_{\ge0}$.
The {\em length} of the multi-index $\alpha$ is $|\alpha|:=
\alpha_1+\cdots+\alpha_d$. For $\alpha=(\alpha_1,\dots,\alpha_d)$,
$\beta=(\beta_1,\dots,\beta_d)$, we write $\beta\le\alpha$ if
$\beta_j\le\alpha_j$ for all $j=1,\dots,d$. The factorial of
$\alpha$ is $\alpha! := \alpha_1!\cdots\alpha_d!$. The binomial
coefficient for multi-indices is
$$
   \dbinom{\alpha}{\beta}:=\dbinom{\alpha_1}{\beta_1}\cdots\dbinom{\alpha_d}{\beta_d}
   =\frac{\alpha!}{\beta!(\alpha-\beta)!};
$$
note that, by definition,
\begin{equation*}
  \dbinom{\alpha}{\beta}=0\qquad \mbox{if $\beta\nleq\alpha$}.
\end{equation*}

By $x^\alpha$, where $x=(x_1,\dots,x_d)\in\mathbb{R}^d$,
$\alpha=(\alpha_1,\dots,\alpha_d)\in\Z^d_{\ge0}$, denote the monomial $x_1^{\alpha_1}\cdots
x_d^{\alpha_d}$. Note that the
{\em total degree} of $x^\alpha$ is $|\alpha|$. The
multi-dimensional version of the {\em binomial formula} is
$$
  (x+y)^\alpha=\sum_{\begin{subarray}{l}
             \beta\in\Z^d_{\ge0}\\\beta\le\alpha
           \end{subarray}}
     \dbinom{\alpha}{\beta}x^\beta y^{\alpha-\beta},\quad
                  \alpha\in\Z_{\ge0}^d,\ x,y\in\R^d.
$$
By $P_L$ denote a polynomial of total degree $L$; and by $P_{\le L}$ denote
a polynomial the total degree of which is less than or equal to $L$.
Denote by $\Pi$ the space of {\em all} polynomials on $\R^d$.
Also denote by
$\Pi_{\le L}$, $L\in\Z_{\ge0}$, the space of polynomials with
total degree less than or equal to $L$:
$\Pi_{\le L}:=\Span\set{x^\alpha}{x\in\R^d, \alpha\in\Z^d_{\ge0}, |\alpha|\le L}$;
and by
$\Pi_L$ the space of polynomials whose
total degree is equal to $L$:
$\Pi_L:=\Span\set{x^\alpha}{x\in\R^d, \alpha\in\Z^d_{\ge0}, |\alpha|= L}$ (here and
in the sequel, the `$\Span$' means the linear span over $\R$).

By $R(x)$ denote a power series $R(x):=\sum_{k\in\Z^d_{\ge0}}a_kx^k$, $x\in\R^d$, $a_k\in\R$.
The {\it order} of the power series $R$ is the least value $|k|$ such that $a_k \ne 0$.
By $R_L(x)$, $L\ge0$, we denote a power series of order $L$ and by $R_{> L}(x)$ denote
a power series of order greater than $L$.

Let $D^\alpha$ stand for the differential operator
$D_1^{\alpha_1}\cdots D_d^{\alpha_d}$, where $D_n$, $n=1,\dots,d$,
is the partial derivative with respect to the $n$th coordinate.
Note that
$D^{(0,\dots,0)}$ is the identity operator.

By $x\cdot y$ denote the inner product of two vectors
$x,y\in\R^d$: $x\cdot y:=x_1y_1+\cdots+x_d y_d$. The Fourier
transform of a function $f\in L^1\left(\R^d\right   )$ is defined by
$$
  F(f)(\xi):=(2\pi)^{-d/2}\int_{\R^d} f(x)e^{-i\xi\cdot x}\,dx=:\hat f(\xi),\qquad \xi\in\R^d.
$$
Note that the Fourier transform can be extended to compactly
supported functions (distributions) from the space
$\mathcal{S}'\left(\R^d\right)$, where $\mathcal{S}'$ denotes the
{\em space of tempered distributions}. So we have the following
useful formula
\begin{equation*}
  F(x^\alpha)(\xi)=i^{|\alpha|} D^\alpha\delta(\xi),
          \qquad \xi\in\R^d,\ \alpha\in\Z^d_{\ge0},
\end{equation*}
where $\delta$ is the Dirac delta-distribution.



\section{Strang--Fix conditions}\label{Sect:WeakenedStrang--FixConditions}


Actually this section is
an auxiliary section
and, for more details, we refer the reader to the forthcoming paper.

\subsection{Definition of matrices}

\subsubsection{Notations}

First we recall some block matrix notions.
We say that a {\em block matrix} is a matrix
broken into
sections called {\em blocks} or {\em submatrices}.
A {\em block diagonal matrix} is a block matrix that is a square matrix such that
the main diagonal square submatrices can be nonzero
and the off-diagonal submatrices are zero matrices.
The (block)
diagonals can be specified by an index $k$
measured relative to the main diagonal, thus the main diagonal has $k=0$
and the $k$-diagonal consists of the entries on the $k$th diagonal above the main diagonal.
Note that all the $k$-diagonal submatrices, except the submatrices on the main diagonal, can be not square matrices.

By symbol `$\lxl$' we denote the {\em lexicographical order} and
by ${\cal A}_L$, $L\in\Z_{\ge0}$, denote the {\em lexicographically ordered set} of all the
multi-indices of length $L$
\begin{equation*}
 {\cal A}_L:=\left(\mathstrut^1\!\alpha,\mathstrut^2\!\alpha,\dots,
           \mathstrut^{d(L)}\!\alpha\right),\qquad
  \begin{aligned}
           &\ind{j}{\alpha}\in\Z^d_{\ge0},\ |\ind{j}{\alpha}|=L,\ j=1,\dots,d(L),\\
           &\ind{j}{\alpha}\lxl\ind{j'}{\alpha}\ \Longleftrightarrow\ j<j',
  \end{aligned}
\end{equation*}
where
\begin{equation*}
    d(L):=\dbinom{d+L-1}{L}
          =\frac{(d+L-1)!}{L!(d-1)!}
\end{equation*}
is the number of $L$-combinations with repetition from the $d$
elements.

By $\cc{\cl A}_L$ we denote a {\em concatenated set
of multi-indices}:
$$
  \cc{\cl A}_L:=\left(\cl A_0,\cl A_1,\dots,\cl A_L\right),
$$
where the comma symbol must be considered as the concatenation operator to
join 2 sets.
Actually the order of the set $\cc{\cl A}_L$ is the graded lexicographical order.
By $\cc d(L)$ denote the length of the concatenated set like $\cc{\cl A}_L$
\begin{equation*}
  \cc d(L):=d(0)+d(1)+\cdots+d(L)=\dfrac{(d+L)!}{L!d!}.
\end{equation*}

By $\cl{P}_L$, $L\in\Z_{\ge0}$,
denote the lexicographically ordered set of the monomials of total degree
$L$
\begin{equation*}
  \cl P_L(x) := \left( x^{\ind{1}{\alpha}},\dots, x^{\ind{d(L)}{\alpha}}\right),\qquad
         x\in\R^d,\ \left(\mathstrut^1\!\alpha,\dots,\mathstrut^{d(L)}\!\alpha\right)={\cal A}_L.
\end{equation*}
For $\beta\in\Z^d_{\ge0}$,
let ${\cal P}^\beta_L$ denote
the following set of the monomials
\begin{equation}\label{P(x)}
  {\cal P}^\beta_L(x):=\left(\dbinom{\ind{1}{\alpha}}{\beta}x^{\ind{1}{\alpha}-\beta},\dots,
     \dbinom{\ind{d(L)}{\alpha}}{\beta}x^{\ind{d(L)}{\alpha}-\beta}\right),\qquad
         \left(\mathstrut^1\!\alpha,\dots,\mathstrut^{d(L)}\!\alpha\right)={\cal A}_L.
\end{equation}
Similarly, define
ordered sets of the differential operators as
\begin{align}
  \cl D_L &:= \left((-i)^{L}D^{\ind{1}{\alpha}},\dots,(-i)^{L}D^{\ind{d(L)}{\alpha}}\right),
          \nonumber\\
  {\cal D}^\beta_L &:= \left((-i)^{L-|\beta|}\dbinom{\ind{1}{\alpha}}{\beta}D^{\ind{1}{\alpha}-\beta},\dots,
              (-i)^{L-|\beta|}\dbinom{\ind{d(L)}{\alpha}}{\beta}D^{\ind{d(L)}{\alpha}-\beta}\right),
           \label{DBetaVector}
\end{align}
where $\left(\mathstrut^1\!\alpha,\dots,\mathstrut^{d(L)}\!\alpha\right)={\cal A}_L$.
Note that if $\beta\not\le\ind{j}{\alpha}$,
then the $j$th entries of~\eqref{P(x)} and~\eqref{DBetaVector}
are zero.
Moreover, if $|\beta|> L$, sets~\eqref{P(x)},~\eqref{DBetaVector}
are zero sets.

By $\cc\cP_L$ and $\cc\cP^\beta_L$
denote concatenated sets of the monomials
\begin{equation}\label{ConcatenatedP(x)}
  \cc{\cl P}_L :=
  \left(\cP_0,\cP_1,\dots,\cP_L\right),\qquad
  {\cc{\cal P}}^\beta_L:=\left(\cP_0^\beta,\cP_1^\beta,\dots,\cP_L^\beta\right).
\end{equation}
The concatenated
sets of the derivatives $\cc \cD_L$, $\cc \cD_L^\beta$ are defined
similarly to~\eqref{ConcatenatedP(x)}.

In the sequel, we shall frequently
enclose the symbols of matrices in the square brackets. In particular, we shall
interpret ordered sets (for example, the sets ${\cal P}_L$, ${\cal
D}_L$) as {\em row-matrices} and enclose their symbols in the square
brackets.

For some $L,l\in\Z_{\ge0}$,
define $d(l)\times d(L)$ matrices $\bs P_L^l$, $\bs D_L^l$
as follows
\begin{equation*}
  \bs P_L^l
  :=
  \begin{bmatrix}
    \left[\cl P^{\ind{1}{\beta}}_L\right] \\[1,5ex] \left[\cl P^{\ind{2}{\beta}}_L\right] \\ \vdots \\
    \left[\cl P^{\ind{d(l)}{\beta}}_L\right]  \\[1.5ex]
  \end{bmatrix},\qquad
  \bs D_L^l
  :=
  \begin{bmatrix}
    \left[\cl D^{\ind{1}{\beta}}_L\right] \\[1.5ex] \left[\cl D^{\ind{2}{\beta}}_L\right] \\ \vdots \\
    \left[\cl D^{\ind{d(l)}{\beta}}_L\right]  \\[1.5ex]
  \end{bmatrix},
\end{equation*}
where $\left(\mathstrut^1\!\beta,\dots,\mathstrut^{d(l)}\!\beta\right)={\cal A}_l$
and $\cl P^{\ind{j}{\beta}}_L$, $\cl D^{\ind{j}{\beta}}_L$ are given by~\eqref{P(x)}
and~\eqref{DBetaVector}, respectively.
Note that, by definition,
if $l>L$, then $\bs P^l_L$, $\bs D^l_L$ are zero matrices.

Define  $\cc d(L)\times d(L)$
matrices
$\bs P_l$, $\bs D_l$, $l=0,1,\dots L$,
as follows
\begin{equation}\label{MatrixD}
  \bs P_l
  :=\begin{bmatrix}
      \bs P^0_l \\[0.5ex] \bs P^1_l \\ \vdots \\ \bs P_l^l \\ 0 \\ \vdots \\ 0 \\
    \end{bmatrix},\qquad
  \bs D_l
  :=\begin{bmatrix}
      \bs D^0_l \\[0.5ex] \bs D^1_l \\ \vdots \\ \bs D_l^l \\ 0 \\ \vdots \\ 0 \\
    \end{bmatrix}.
\end{equation}
Also define concatenated
$\cc d(L)\times\cc d(L)$ matrices $\cc{\bs P}_L$, $\cc{\bs D}_L$ as
\begin{align}
  \cc{\bs P}_L
  &:=\begin{bmatrix}
      \bs P_0 & \bs P_1 & \dots & \bs P_{L-1} & \bs P_L \\
     \end{bmatrix}
  =\begin{bmatrix}
      \bP^0_0 & \bP^0_1 & \dots & \bs P^0_{L-1} & \bs P^0_L \\
      0 & \bP^1_1 & \dots & \bP^1_{L-1} & \bP^1_L \\
      \vdots & \vdots & \ddots & \vdots & \vdots \\
      0 & 0 & \dots  & \bP^{L-1}_{L-1} & \bP^{L-1}_L \\
      0 & 0 & \dots & 0 &\bP_L^L\\
    \end{bmatrix},
    \label{ExtendedMatrixP}\\
  \cc{\bs D}_L
  &:=\begin{bmatrix}
      \bs D_0 & \bs D_1 & \dots & \bs D_{L-1} & \bs D_L \\
     \end{bmatrix}
  =\begin{bmatrix}
      \bD^0_0 & \bD^0_1 & \dots  & \bs D^0_{L-1} & \bs D^0_L \\
      0       & \bD^1_1 & \dots  & \bD^1_{L-1}   & \bD^1_L \\
      \vdots & \vdots & \ddots & \vdots & \vdots \\
      0 & 0   & \dots  & \bD^{L-1}_{L-1} & \bD^{L-1}_L \\
      0 & 0   &\dots  & 0 &\bD_L^L\\
    \end{bmatrix}.
    \label{ExtendedMatrixD}
\end{align}
Note that, in formulas~\eqref{MatrixD}--\eqref{ExtendedMatrixD},
the `$0$' symbols must be considered as zero matrices of the corresponding sizes.

\begin{remark}
The component-wise form of the matrix $\cc\bP_L$ is
\begin{equation}\label{Component-WiseP}
  \left[\cc\bP_L(x)\right]_{jk}
  =
  \left\{%
  \begin{aligned}
    &\dbinom{\ind{k}{\alpha}}{\ind{j}{\beta}}x^{\ind{k}{\alpha}-\ind{j}{\beta}}, &&
                                                                    \ind{j}{\beta}\le\ind{k}{\alpha}\,,\\
    &\quad0,                                        && \mbox{otherwise},
  \end{aligned}%
  \right.\qquad
  \begin{aligned}
    &j,k=1,2,\dots,\cc d(L),\\
    &\left(\ind{1}{\alpha},\dots,\ind{\cc d(L)}{\alpha},\right)\\
    & = \left(\ind{1}{\beta},\dots,\ind{\cc d(L)}{\beta},\right)=\cc\cA_L;
  \end{aligned}
\end{equation}
and similarly for the matrix $\cc\bD_L$.
\end{remark}

Note that the relations between the matrices of monomials and derivatives are
\begin{align*}
  &\cl D_L = \cl P_L(-iD),\quad \cl D_L^\beta = \cl P_L^\beta(-iD),\quad
    {\bs D}^l_L={\bs P}^l_L(-iD),\\
  &{\bs D}_L={\bs P}_L(-iD),\quad
    \cc{\bs D}_L=\cc{\bs P}_L(-iD).
\end{align*}

Finally define infinite-rows matrices:
\begin{equation}\label{InfiniteRowsMatrix}
  \bs\Delta_L\hat f :=
  \begin{bmatrix}
    \vdots \\ \bs D_L\hat f(2\pi\,\ind{j-1}{n}) \\ \bs D_L\hat f(2\pi\,\ind{j}{n})\\
    \bs D_L\hat f(2\pi\,\ind{j+1}{n}) \\ \vdots
  \end{bmatrix},
  \qquad
  \cc{\bs\Delta}_L\hat f :=
  \begin{bmatrix}
    \vdots \\ \cc{\bs D}_L\hat f(2\pi\,\ind{j-1}{n}) \\ \cc{\bs D}_L\hat f(2\pi\,\ind{j}{n})\\
    \cc{\bs D}_L\hat f(2\pi\,\ind{j+1}{n}) \\ \vdots
  \end{bmatrix},
\end{equation}
where the $d$-tuples $\ind{j}{n}\in\Z^d\setminus\{0\}$ are arbitrarily ordered.

\subsubsection{Ranks of the matrices $\bs D_L^l$}

In this subsection, we investigate the ranks of the matrices $\bs D^l_L$,
$l=0,\dots,L$.

First consider the matrix $\bs D_L^L$. $\bs D_L^L$ is a square
$d(L)\times d(L)$ matrix. It easy to see that the $j$th row of
the matrix $\bs D_L^L\hat f$ contains only one nonzero
element $\hat f$, which is situated on the $j$th position.
Thus $\bs D_L^L\hat f=\bs I\hat f$, where $\bs I$ is the
corresponding identity matrix.

Secondly 
$\bs D^0_L\hat f$ is a
row-matrix $\left[D^{\ind{1}{\alpha}}\hat f\  \cdots\  D^{\ind{d(L)}{\alpha}}\hat f\right]$,
where
$\left(\ind{1}{\alpha},\dots,\ind{d(L)}{\alpha}\right)={\cA}_L$.
Consequently the matrix $\bs D^0_L\hat f$ has the nonzero rank iff
there exists at least one multi-index $\ind{j}{\alpha}\in\cA_L$
such that $D^{\ind{j}{\alpha}}\hat f\ne0$.

Now consider the matrices $\bs D^l_L$, $l=1,\dots,L-1$. Note
that $\bs D^l_L$ is a $d(l)\times d(L)$ matrix.

\begin{theorem}
Let $L\in\N$. The $d(l)\times d(L)$ matrix $\bs D^l_L\hat f$,
$l=1,\dots,L-1$, has the full rank, i.\,e., the rank of $\bs D^l_L\hat f$ is
equal to $d(l)$, if and only if there exists at least one nonzero
derivative $D^{\gamma}\hat f$, $|\gamma|=L-l$.
\end{theorem}
To prove this theorem, we refer the reader to the forthcoming paper.

Hence we see that each of the matrices $\bs D^l_L
\hat f$,
$l=0,\dots,L$, is either a full rank
matrix or zero matrix.


\subsection{Definitions and Theorems}

\subsubsection{General case}

We begin with definitions.
Following to C.~de~Boor, see~\cite{deBoor},
we change slightly the definition of space~\eqref{SDefinition}.
\begin{definition}
By definition, put
\begin{equation}\label{SDefinitionNew}
  {\cV}:=\set{P}{P\in\Pi,\,\left(P(-iD)\hat f\right)(2\pi k)=0,\ k\in\Z^d\setminus\{0\}},
\end{equation}
where $\Pi$ is the space of all polynomials on $\R^d$
and $f\in S'(\R^d)$ is a compactly supported function.
By
$\cV_{\rm sh}$ we denote the {\em largest shift-invariant subspace} of the space $\cV$ given by~\eqref{SDefinitionNew}.
\end{definition}

Let us recall the C.~de~Boor statement on a polynomial space in the span of integer shifts of a compactly
supported function.
\begin{theorem}[C.~de~Boor~{\cite[Proposition~2.2]{deBoor}}]
Suppose a function $f\in S'(\R^d)$ is compactly supported, then
$$
  \Pi\cap\Span\left\{f(\cdot-k),k\in\Z^d\right\}=\cV_{\rm sh}.
$$
\end{theorem}

\begin{definition}
Suppose $A$ is an $n\times m$ matrix.
Put
\begin{equation*}
  \ker A:=\set{v\in\R^{m}}{Av=0}.
\end{equation*}
We say that the linear space $\ker A$
is the (right) {\em null-space of the matrix} $A$.
\end{definition}

Here and in the sequel, we shall always consider the vector in the matrix-vector multiplication as
a column-matrix. Moreover,
we shall also suppose that the matrix-vector multiplication
is distributive over a (countable or even uncountable) set of vectors (points) of
an Euclidian space: $A\set{s}{s\in\R^d}:=\set{As}{s\in\R^d}$,
where $A$ is a $d\times d$ matrix.

\begin{definition}\label{Definition:LDefinition}
Let a function $f\in S'(\R^d)$ be compactly supported.
Let matrices $\cc\Delta_l\hat f$, $l\in\Z_{\ge0}$, be given by~\eqref{InfiniteRowsMatrix}.
By definition,
put
\begin{equation}\label{LDefinition}
  L:=\min\set{l\in\Z_{\ge0}}{%
    \dim\ker\cc\Delta_l\hat f=\max\set{\dim\ker\cc\Delta_m\hat f}{m\in\Z_{\ge0}}},
\end{equation}
i.\,e., $L$ is
the minimal $l\in\Z_{\ge0}$ such that $\dim\ker\cc\Delta_l\hat f$
is maximal;
then we say that $L$ is the
{\em order of the Strang--Fix conditions}.
\end{definition}

\begin{definition}
By $V$ denote the null-space of the matrix $\cc\Delta_L\hat f$:
\begin{equation}\label{VDefinition}
  V:=\ker\cc\Delta_L\hat f,
\end{equation}
where $L$ is the order of the Strang--Fix conditions defined by~\eqref{LDefinition}.
\end{definition}

Let $L\ge0$.
Suppose $V\subseteq\R^{\cc d(L)}$ is a linear space
(for example, given by~\eqref{VDefinition}); then $V$ always can be decomposed as follows
\begin{equation}\label{RDecomposition}
  V= V^0\oplus V^1\oplus\cdots\oplus V^{L},
\end{equation}
where the subspace $V^l$, $l=0,\dots,L$, corresponds to the subset ${\cD}_l$
of the set $\cc\cD_{L}$; i.\,e., if a vector $v\in V$ belongs to $V^l$, then
$v$ is necessarily of the form $v=\left(0,\dots, 0,\,v_{\cc d(l-1)+1},\dots,v_{\cc d(l)},\,0, \dots,0\right)$
(where $\cc d(-1):=0$).

Now we can reformulate the definition of the Strang--Fix conditions order, see~\eqref{LDefinition}.

\begin{definition}
Let a function $f\in S'(\R^d)$ be compactly supported. Let $L\ge0$ and let $V:=\ker\cc\Delta_L\hat f$.
Suppose $V$ is decomposed like~\eqref{RDecomposition};
then $L$ is the order of the Strang--Fix conditions iff
all the subspaces $V^l$, $l=0,1,\dots,L$, are nonzero and $V^{L+1}$
is the zero space.
\end{definition}

To proof the fact that this definition coincides with Definition~\ref{Definition:LDefinition},
we refer the reader to the forthcoming paper.

\begin{definition}
Let $L$ be the order of the Strang--Fix conditions.
Let a linear space $V$ be given by~\eqref{VDefinition} and the set $\cc\cP_L$ be given
by~\eqref{ConcatenatedP(x)}; then by $\te\cV$ we denote the following polynomial space
\begin{equation}\label{cVDefinition}
  \te\cV:=\set{\left[\cc\cP_L\right]v}{v\in V}.
\end{equation}
\end{definition}

Below we present the main theorem of this section.

\begin{theorem}
Let a function $f\in S'(\R^d)$ be compactly supported. Let $L$ be the order of the Strang--Fix conditions,
let the polynomial spaces $\cV$ and $\te\cV$ be given by~\eqref{SDefinitionNew} and~\eqref{cVDefinition}, respectively;
then we have
$$
  \te\cV=\cV_{\rm sh}.
$$
\end{theorem}

\begin{lemma}\label{Lemma:P(x+y)}
Let $L\ge0$ and let the set $\cc\cP_L$ be given by~\eqref{ConcatenatedP(x)},
then we have
\begin{equation}\label{Pxh}
  \left[\cc\cP_L(x+y)\right] = \left[\cc\cP_L(x)\right]\left[\cc\bP_L(y)\right],
\end{equation}
where $x,y\in\R^d$ and the matrix $\cc\bP_L$ is given by~\eqref{ExtendedMatrixP}.
\end{lemma}

We omit the proof of the lemma
and note only that formula~\eqref{Pxh} is the direct consequence of the binomial formula and
formula~\eqref{Component-WiseP}.

\begin{proof}[Proof of the theorem]
Suppose a polynomial $P$ belongs to $\te\cV$; then, by~\eqref{cVDefinition},
$P(x)=\left[\cc\cP_L(x)\right]v$ for some $v\in V$.
Thus, for any $n\in\Z^d\setminus\{0\}$, we have
$\left(P(-iD)\hat f\right)(2\pi n)=\left[\cc\cD_L\hat f(2\pi n)\right]v$.
Since, for any $n\in\Z^d\setminus\{0\}$, the row-matrix $\left[\cc\cD_L\hat f(2\pi n)\right]$
is the first row of the corresponding matrix $\cc\bD_L\hat f(2\pi n)$
(see~\eqref{ExtendedMatrixD}) and $v\in V$; we have $\left[\cc\cD_L\hat f(2\pi n)\right]v=0$. Hence
$\left(P(-iD)\hat f\right)(2\pi n)=0$, $\forall n\in\Z^d\setminus\{0\}$; consequently $P\in\cV$.

For an arbitrary {\em shift} $h\in\R^d$, consider the polynomial $P(x+h)$. Using Lemma~\ref{Lemma:P(x+y)}, we have
$P(x+h)=\left[\cc\cP_L(x+h)\right]v
  = \left[\cc\cP_L(h)\right]\left[\cc\bP_L(x)\right]v$, where
$\cc\bP_L$ is matrix~\eqref{ExtendedMatrixP}. Thus, for any $n\in\Z^d\setminus\{0\}$,
we obtain $P(-iD+h)\hat f(2\pi n)=\left[\cc\cP_L(h)\right]\left[\cc\bD_L\hat f(2\pi n)\right]v=0$.
Consequently $\te\cV\subseteq\cV_{\rm sh}$.

Now we prove the contrary sentence: $\cV_{\rm sh}\subseteq\te\cV$.
Let $P\in\cV_{\rm sh}\subseteq\cV$, then there exists a vector $v\in\R^{\cc d(L)}$ such that
$P(x)=\left[\cc\cP_L(x)\right]v$.

Assume the converse: $P\not\in\te\cV$; then $v\not\in V$. Consequently
there are exist at least one point $n\in\Z^d\setminus\{0\}$ and
multi-index $\beta\in\Z^d_{\ge0}$,
$|\beta|\le L$, such that $\left[\cc\cD_L^\beta\hat f(2\pi n)\right]v\ne0$. Suppose
$\beta\ne\left(0,\dots,0\right)$, then
there exists a shift $h\in\R^d$ such that $\left(P(-iD+h)\hat f\right)(2\pi n)\ne0$. Thus
$P\not\in\cV_{\rm sh}$. (If $\beta=\left(0,\dots,0\right)$, then $\left[\cc\cD_L\hat f(2\pi n)\right]v\ne0$
for some $n\in\Z^d\setminus\{0)\}$; and consequently $\left(P(-iD)\hat f\right)(2\pi n)\ne0$. Hence $P\not\in\cV$.)
This contradiction proves that $\cV_{\rm sh}\subseteq\te\cV$. This completes the proof.
\end{proof}

Finally we state and prove an auxiliary theorem, which will be useful later.

\begin{theorem}\label{Theorem:PolynomialFromKernel}
Let $P(x)$, $x\in\R^d$, be an algebraic polynomial of total degree $L\ge 0$.
Let the matrix $\cc\bD_L$
be given by~\eqref{ExtendedMatrixD}. Then an
algebraic polynomial $\left[\cc\cP_L\right] v$,
$v\in\R^{\cc d(L)}$,
belongs to $\ker P(-iD)$ iff $v\in\ker \cc\bD_L P(0)$.
\end{theorem}
\begin{proof}
Suppose $v\in\ker \cc\bD_L P(0)$ and
consider an expression
$$
  P(-iD)\left[\cc\cP_L(\xi)\right] v,\quad \xi\in\R^d.
$$
Taking the inverse Fourier transform of the previous expression, we obtain the following expression
$P(x)\left[\cc\cD_L\delta(x)\right]v$, which is equivalent to $\left[\cc\cD_L P(0)\right]v$.
Since $\left[\cc\cD_L P(0)\right]$ is the first row of the matrix $\cc\bD_L P(0)$,
we have $\left[\cc\cD_L P(0)\right]v=0$. Thus the algebraic polynomial $\left[\cc\cP_L\right]v$
belongs to $\ker P(-iD)$.

Contrary, suppose that, for a vector $v\in\R^{\cc d(L)}$, the polynomial $\left[\cc\cP_L\right] v$
belongs to $\ker P(-iD)$. Consequently we have
$P(-iD)\left[\cc\cP_L\right]v=0$. Since the differentiation commutes with the translation, it follows
that the previous relation is valid for any shift of the argument. Thus $P(-iD)\left[\cc\cP_L(\cdot+h)\right]v=0$,
$\forall h\in\R^d$.
Taking the inverse Fourier transform of the previous relation and using Lemma~\ref{Lemma:P(x+y)},
we have $P(x)\left[\cc\cP_L(h)\right]\left[\cc\bP_L(-iD)\delta(x)\right]v=0$ or
$$
  \left[\cc\cP_L(h)\right]\left[\cc\bD_L P(0)\right]v=0.
$$
Since the previous relation is valid for an arbitrary $h\in\R^d$, we have $v\in\ker\cc\bD_LP(0)$.
\end{proof}

\begin{remark}
Theorem~\ref{Theorem:PolynomialFromKernel} is valid for any total degree $L\ge0$ of the polynomials; and the theorem
supplies {\em all} the algebraic polynomials up to degree $L$ from the kernel of a given differential operator
with a polynomial symbol.
Note also that the matrix $\cc\bD_L P(0)$ has a non-zero null-space iff the constant term of the
polynomial $P$ (the identity term of the operator $P(-iD)$) vanishes.
\end{remark}

\subsubsection{Affinely invariant case}

First
we present some obvious statement about scale-invariant polynomial spaces.

\begin{statement}\label{State:AffinelyInvariance}
Any polynomial space $\cV$ is scale-invariant iff the space can be decomposed
as
$$
  {\cal V}=\bigoplus\limits_k\left({\cal V}\cap\Pi_k\right).
$$
\end{statement}

Using Statement~\ref{State:AffinelyInvariance}, we can formulate the following theorem.

\begin{theorem}\label{Theorem_Scale-Invariance}
Let a function $f\in S'(\R^d)$ be compactly supported.
Let $L$ be the order of the Strang--Fix conditions.
Let a linear space $V$ be the null-space
of the matrix $\cc\Delta_L\hat f$,
and
let a polynomial space $\te\cV$ be given by~\eqref{cVDefinition}.
Then the polynomial space $\te\cV$ is affinely invariant iff, for any $v\in V$,
we have $\left(v_{\cc d(l-1)+1}, \dots,v_{\cc d(l)}\right)\in \ker\Delta_l\hat f$,
$l=0,1,\dots,L$, (\,$\cc d(-1):=0$).
\end{theorem}

The proof is trivial.

\begin{remark}
If the null-space $V:=\ker\cc\Delta_L\hat f$ is not scale-invariant; then
it is always possible to consider an affinely invariant subspace of $V$ and to define
the corresponding affinely invariant polynomial subspace.
Note also that
the order of the Strang--Fix conditions defined by the largest affinely invariant
subspace of $V$ can be less than the order corresponding to the initial space $V$.
\end{remark}

\begin{remark}
In the paper, we consider the polynomial spaces (and the spans that contain the polynomial spaces)
only over $\R$. Nevertheless it is possible to extend our consideration to the field $\C$, but
the corresponding generalization is left to the reader.
\end{remark}


\section{Elliptic scaling functions}\label{Sect:EllipticScalingFunctions}


\subsection{Preliminaries and notations}

Let us recall here some notation and formulas on scaling functions, see, for example,~\cite{Dau,NPS}.

In general, a scaling function $\phi$ satisfies a {\em refinement relation}
\begin{equation}\label{ScalingRelation}
  \phi(x) = \sum_{  k\in{\Z\mathstrut}^d}h_{k}|\det A|^\frac{1}{2}\phi(Ax - k),
     \qquad x\in\R^d,
\end{equation}
where $A$ is a $d\times d$ matrix
and the matrix is called the {\em dilation matrix}.
In the paper,
we shall suppose that the dilation matrices are
real integer matrices whose eigenvalues are greater
than 1 in absolute value.
Then, for any dilation matrix $A$, we have
\begin{equation}\label{DilationMatrixProperty}
  \lim_{j\to\infty}\left|A^{j}x\right|\to\infty,\qquad \forall x\in\R^d,\  x\ne 0.
\end{equation}

Scaling relation~\eqref{ScalingRelation}
can be rewritten in the Fourier domain as
\begin{equation*}
  \hat\phi(\xi) = m_0\left(\left(A^T\right)^{-1}\xi\right)
     \hat\phi\left(\left(A^T\right)^{-1}\xi\right),     \qquad \xi\in\R^d,
\end{equation*}
where $m_0(\xi)$, $\xi\in\R^d$, is a $2\pi$-periodic function, which is called the {\em mask}.
The Fourier transform of the scaling function $\phi$ can be defined by the mask $m_0$ as follows
\begin{equation}\label{HatPhi}
  \hat\phi(\xi)=\prod_{j=1}^\infty
    m_0\left(\left(A^{-T}\right)^{j}\xi\right).
\end{equation}
To simplify the notations,
by $A^{-T}$ we denote the matrix
$(A^T)^{-1}\equiv(A^{-1})^T$.

In this paper, we shall also use the so-called {\em nonstationary
scaling functions}, see~\cite{BerkolaikoNovikov,BVR,NPS,VBU}.
In the nonstationary case, the scaling functions
of different scales are not the scaled versions of a
single function;
and the masks are also different for different scales.
Now the refinement relation (in the Fourier domain) is of the form
\begin{equation*}
  \ssc{n}{\hat\phi}(\xi) = \ssc{n+1}{m_0}\left(A^{-T}\xi\right)
     \ssc{n+1}{\hat\phi}\left(A^{-T}\xi\right),
 \end{equation*}
and formula~\eqref{HatPhi} becomes
\begin{equation*}
  \ssc{n}{\hat\phi}(\xi)=\prod_{j=1}^\infty
    \ssc{n+j}{m_0}\left(\left(A^{-T}\right)^{j}\xi\right),
\end{equation*}
where the superscripts denote the types of the scaling functions (and the corresponding masks).

Below we present some well-known statement concerning the dilation matrices.

\begin{statement}
Let $A$ be a non-singular $d\times d$ matrix with integer entries. Then the number
of the cosets of $\Z^d$ by modulo $A$ is equal to $|\det A|$ and
the set $\Z^d\cap A[0,1)^d$ is a set of representatives of the quotient $\Z^d/A\Z^d$.
\end{statement}

\begin{definition}
Let $\cS(A)\subset[0,1)^d$ be a set of points such that $A\cS(A)\subset\Z^d$, i.\,e.,
$A\cS(A)$ is a set of representatives of the quotient $\Z^d/A\Z^d$.
\end{definition}

Present also an auxiliary but very useful theorem.

\begin{theorem}\label{Theorem_UnionSets}
Let $A$ be a $d\times d$ {\em dilation} matrix with integer entries,
then we have
\begin{align*}
  &\bigcup_{j=1}^\infty\bigcup_{s\in \cS(A)\setminus\{0\}} A^j\set{k+s}{k\in\Z^d} = \Z^d\setminus\{0\},\\[1.5ex]
  &\begin{aligned}
    &A^j\set{k+s}{k\in\Z^d}\\
    &\qquad \cap A^{j'}\set{k+s'}{k\in\Z^d}=\emptyset
   \end{aligned}
    \qquad\mbox{if}\quad\
  \begin{aligned}
    &j,j'\in\N, j\ne j'\mbox{ or }\\ &s,s'\in\cS(A)\setminus\{0\},s\ne s'.
  \end{aligned}
\end{align*}
\end{theorem}
\begin{proof}
Denote $\set{k+s}{k\in\Z^d}$, where $s\in\cS(A)$, by $\Xi_s$.
Note that $\Xi_{(0,\dots,0)}=\Z^d$.
For all $J\in\N$, we shall prove the following relations
\begin{equation}\label{ToBeProved}
  \begin{aligned}
    &\left(\bigcup_{j=1}^{J}\bigcup_{s\in\cS(A)\setminus\{0\}}A^{j}\Xi_s\right)\cup
       A^{J}\Z^d = \Z^d,\\[0.7ex]
    &\begin{aligned}
       &A^j\Xi_s \cap A^{j'}\Xi_{s'}=\emptyset\quad\mbox{if }\\
       &j,j'=1,\dots,J,\ j\ne j'\\
       &\mbox{or }\ s,s'\in\cS(A)\setminus\{0\},\ s\ne s'
    \end{aligned}\quad\ \mbox{and}\quad
    \left(\bigcup_{j=1}^{J}\bigcup_{s\in\cS(A)\setminus\{0\}}A^{j}\Xi_s\right)\cap
      A^{J}\Z^d=\emptyset.
  \end{aligned}
\end{equation}
The proof is by induction over $J$.
For $J=1$, since $\set{A\Xi_s}{s\in\cS(A)}$ are disjoint cosets of $\Z^d$, we have
$$
  \left(\bigcup_{s\in\cS(A)\setminus\{0\}}A\Xi_s\right)\cup A\Z^d=\Z^d,
  \qquad A\Xi_s\cap A\Xi_{s'} =\emptyset\ \mbox{ if }\  s,s'\in\cS(A), s\ne s'.
$$
By the inductive assumption, we have the validity of relations~\eqref{ToBeProved}.
Since $\set{A^{J+1}\Xi_s}{s\in\cS(A)}$ are disjoint cosets of $A^J\Z^d$;
it follows that
$$
  A^J\Z^d=\left(\bigcup_{s\in\cS(A)\setminus\{0\}}A^{J+1}\Xi_s\right)\cup
    A^{J+1}\Z^d,\qquad\ \
   \begin{aligned}
     &A^{J+1}\Xi_s\cap A^{J+1}\Xi_{s'} =\emptyset\\
     &\mbox{if }\  s,s'\in\cS(A), s\ne s'.
   \end{aligned}
$$
Thus we obtain the relations
$$
  \begin{aligned}
    &\left(\bigcup_{j=1}^{J+1}\bigcup_{s\in\cS(A)\setminus\{0\}}A^{j}\Xi_s\right)\cup
       A^{J+1}\Z^d = \Z^d,\\[0.7ex]
    &\begin{aligned}
       &A^j\Xi_s \cap A^{j'}\Xi_{s'}=\emptyset\quad\mbox{if }\\
       &j,j'=1,\dots,J+1,\ j\ne j'\\
       &\mbox{or }\ s,s'\in\cS(A)\setminus\{0\},\ s\ne s'
    \end{aligned}\quad\ \mbox{and}\quad
    \left(\bigcup_{j=1}^{J+1}\bigcup_{s\in\cS(A)\setminus\{0\}}A^{j}\Xi_s\right)\cap
      A^{J+1}\Z^d=\emptyset.
  \end{aligned}
$$
By induction, expressions~\eqref{ToBeProved} are valid for all $J\in\N$.

Now we must consider the limit of $A^j\Z^d$ as $j\rightarrow\infty$.
Denote by $C_r$ the circle of radius $r$ with the center at the origin; then,
using~\eqref{DilationMatrixProperty}, we see that, for
an arbitrary large $r\in\R$,
there exists a number $J\in\N$ such that for all $j'>J$ we have
$\left|A^{j'}k\right|>r$, $k\in\Z^d\setminus\{0\}$, and consequently
$C_r\cap\left(\bigcup_{s\in\cS(A)\setminus\{0\}}A^{j'}\Xi_s\right)=\emptyset$. Thus
$C_r\cap\Z^d\setminus\{0\}\subset\left(\bigcup_{j=1}^J\bigcup_{s\in\cS(A)\setminus\{0\}}A^j\Xi_s\right)$.
Tending the radius $r$ to the infinity, we
extend the relation was to be proved to all $\Z^d\setminus\{0\}$.
\end{proof}

\subsection{Isotropic matrices decomposition}

Let us recall the definition of an isotropic matrix.
\begin{definition}
Any square matrix is called {\em isotropic} if
the matrix is diagonalizable over $\C$ and
all its eigenvalues are equal in absolute value,
see, for example,~\cite{R.-Q.Jia}.
\end{definition}

\begin{theorem}\label{Theorem_MyDecomposition}
Let $\te A$ be a $d\times d$ real isotropic matrix and
let $|\det\te A|=1$,
then
\begin{equation}\label{MyDecomposition}
  \te A=QUQ^{-1},
\end{equation}
where $U$ is an orthogonal matrix and $Q$ is a positive definite symmetric matrix.
\end{theorem}
\begin{proof}
Since the matrix $\te A$ is isotropic,
the matrix is diagonalizable:
$$
  \te A =T{\Lambda}T^{-1},
$$
where $\Lambda$ is a diagonal matrix.
Using the polar decomposition, we can always present $T$ as follows
\begin{equation}\label{PolarDecomposition}
  T=QF,
\end{equation}
where $F$ is a unitary matrix and $Q$ is a positive definite Hermitian matrix:
$Q^2=TT^*$ (where the matrix $T^*$ is the Hermitian conjugate of the matrix $T$).

Now we shall prove that $Q$ is a real matrix.
Since the matrix $\te A$ is real; it follows that if the dimension $d$ is even,
then the eigenvalues of $\te A$ are encountered
as pairs $\{\lambda_j,\cc\lambda_j\}$
(here and in the sequel, the overline stands for the complex conjugation),
$j=1,\dots,d/2$ (taking into account the multiplicities of the eigenvalues); else, if $d$ is odd,
in addition to the pairs of eigenvalues, there is an eigenvalue $\pm1$.
Note also that, among the eigenvalues $\lambda_j$, $\cc \lambda_j$, the real eigenvalues $\pm 1$ can be.
Since the eigenvectors corresponding to the real eigenvalues are real, these
eigenvectors do not influence on the matrix to be complex.
Thus without loss of generality we shall consider the case of an even $d$ only
and suppose that there are not real eigenvalues.

Let $\ind{j}{x}:=\left(\ind{j}{x_1},\ind{j}{x_2},\dots,\ind{j}{x_d}\right)$ be an eigenvector of $\te A$
corresponding to an eigenvalue $\lambda_j$, $j=1,\dots,d/2$;
then without loss of generality the matrix $T$ can be of the form
$$
  T:=
    \begin{bmatrix}
      \ind{1}{x_1} & \cc{\ind{1}{x_1}} & \ind{2}{x_1} & \cc{\ind{2}{x_1}}
                                   &\ldots& \ind{d/2}{x_1} & \cc{\ind{d/2}{x_1}}\\
      \ind{1}{x_2} & \cc{\ind{1}{x_2}} & \ind{2}{x_2} & \cc{\ind{2}{x_2}}
                                   &\ldots& \ind{d/2}{x_2} & \cc{\ind{d/2}{x_2}}\\
      \vdots       &  \vdots           &\vdots         &\vdots
                                   &\ddots&\vdots      &\vdots\\
      \ind{1}{x_d} & \cc{\ind{1}{x_d}} &\ind{2}{x_d} & \cc{\ind{2}{x_d}}
                                   &\ldots& \ind{d/2}{x_d} & \cc{\ind{d/2}{x_d}}\\
    \end{bmatrix}.
$$
Consequently the matrix $\cc T$ can be presented as follows
\begin{equation}\label{TComplexConjugated}
  \cc T=TC,\qquad
   C:=
   \begin{bmatrix}
     c & 0 & \ldots & 0  \\
     0 & c & \ldots & 0  \\
     \vdots&\vdots&\ddots&\vdots\\
     0 & 0 & \ldots  & c \\
       \end{bmatrix},
\end{equation}
where
$c:=\begin{bmatrix}
  0 & 1 \\
  1 & 0 \\
\end{bmatrix}$
is a permutation matrix and the `0' symbol must be interpreted as the $2\times2$ zero matrix.
Using~\eqref{TComplexConjugated}, and since $c^2$
is the $2\times2$ identity matrix,
we have $\cc{Q^2} = \cc{T}\left(\cc{T}\right)^* = TC^2T^* = TT^* = Q^2$.
Consequently $Q^2$ is a {\em real} matrix,
hence the ``square root'' $Q$ is also a real matrix.

Using~\eqref{PolarDecomposition}, the matrix $\te A$ can be written as follows
$$
  \te A = QF{\Lambda}F^{-1}Q^{-1}=QUQ^{-1},
$$
where $U:=F{\Lambda}F^{-1}=F{\Lambda}F^*$.
The matrix $U$ is a unitary matrix. Indeed,
\begin{multline*}
  UU^* = F{\Lambda}F^*F{\Lambda}^*F^*=F
   \begin{bmatrix}
     |\lambda_1|^2 & 0 & \ldots & 0 & 0 \\
     0 & |\cc{\lambda_1}|^2 & \ldots & 0 & 0 \\
     \vdots&\vdots&\ddots&\vdots&\vdots\\
     0 & 0 & \ldots & |\lambda_{d/2}|^2 & 0 \\
     0 & 0 & \ldots & 0 & |\cc{\lambda_{d/2}}|^2 \\
       \end{bmatrix}
   F^*\\ = FIF^* = I,
\end{multline*}
where $I$ is the identity matrix.

Finally since the matrices $\te A$ and $Q$ in decomposition~\eqref{MyDecomposition}
are real, the matrix $U$ must be real (consequently orthogonal) also.
\end{proof}

The following corollary of Theorem~\ref{Theorem_MyDecomposition}
will play an important role in the sequel.

\begin{corollary}\label{CorollaryFromMyDecomposition}
Under the conditions of Theorem~\ref{Theorem_MyDecomposition}, we have
\begin{align}
    \label{Q-2Invariance}
    &\te AQ^{2}\te A^T=\te A^{-1}Q^{2}\te A^{-T}=Q^{2},\\
    \label{Q2Invariance}
    &\te A^TQ^{-2}\te A=\te A^{-T}Q^{-2}\te A^{-1}=Q^{-2}.
\end{align}
\end{corollary}
The proof is straightforward.

\subsection{Quadratic form definition}

Consider a real dilation matrix $A$. Suppose that $A$ is isotropic;
then, using Theorem~\ref{Theorem_MyDecomposition},
$A^{-T}$ can be factored as follows
\begin{equation}\label{MatrixBDecomposition}
    A^{-T}= \frac{1}{q^{1/d}} Q^{-1}UQ,
\end{equation}
where $q:=|\det A|$, $U$ is an orthogonal matrix,
and $Q$ is a symmetric positive definite matrix.

Now we can define a quadratic form
\begin{equation}\label{QuadraticForm}
  W(x):={x}^TQ^{2}x,\quad x\in\R^d.
\end{equation}
Since $Q^{2}$
is positive definite; therefore,
quadratic form~\eqref{QuadraticForm}
is also positive definite.
By Corollary~\ref{CorollaryFromMyDecomposition}, we see that the quadratic
form $W(x)$
is invariant (up to a constant factor)
under the variable transformation
by the matrix $A^{-T}$: $x\mapsto x':= A^{-T}x$.
Indeed, using~\eqref{Q-2Invariance}, we have
\begin{equation}\label{QuadrFormInvariance}
  W(x')=W(A^{-T}x)=x^T A^{-1} Q^{2}A^{-T}x = \frac{1}{q^{2/d}}x^TQ^{2}x
  = \frac{1}{q^{2/d}}W(x).
\end{equation}
(Similarly, the quadratic form ${x}^TQ^{-2}x$ will be invariant under the transformation
by the matrix $A$, see~\eqref{Q2Invariance}.)

\begin{remark}
The matrix $Q$ in formulas~\eqref{MyDecomposition},~\eqref{MatrixBDecomposition}
(consequently, quadratic form~\eqref{QuadraticForm}) is
defined within a constant factor.
\end{remark}

\subsection{Construction of the mask}

Let $A$ be an isotropic dilation matrix and let $A^{-T}$ be factored by formula~
\eqref{MatrixBDecomposition}.
Let ${G}(\xi)$ be a trigonometric polynomial such that its Taylor series about zero begins
with quadratic form~\eqref{QuadraticForm},
i.\,e.,
\begin{equation}\label{m_1SimP}
  {G}(\xi) := W(\xi)+\mbox{ higher order terms}, \qquad \xi\in\R^d.
\end{equation}

Define a mask $m_0$ as follows
\begin{equation}\label{m_0Definition}
  m_0(\xi):=\dfrac{\prod\limits_{s\in{\cS}(A^T)\setminus\{0\}}{G}(\xi+2\pi s)}
     {\prod\limits_{s\in{\cS}(A^T)\setminus\{0\}}{G}(2\pi s)}.
\end{equation}
In formula~\eqref{m_0Definition}, we suppose that
\begin{equation}\label{NonZeroConditionsOnG}
  G(2\pi s)\ne0,\quad\forall s\in\cS(A^T)\setminus\{0\}.
\end{equation}

\begin{definition}
The scaling function corresponding to an isotropic dilation matrix $A$ and mask~\eqref{m_0Definition},
where $G$ is given by~\eqref{m_1SimP} and the quadratic form $W$ is given by~\eqref{QuadraticForm},
is called the {\em elliptic scaling function} (see~\cite{ZakhConstrApprox}).
\end{definition}

Let the matrix $Q^2$ be presented in  component-wise form as follows
\begin{equation*}
  Q^2 := \left[q_{ij}\right]_{i,j=1,\dots,d},\quad q_{ij}\in\R^d,\ \ q_{ij}=q_{ji}.
\end{equation*}
Then quadratic form~\eqref{QuadraticForm} is
\begin{equation}\label{AnotherQuadraticForm}
  W(\xi)
      :=\sum_{1\le i\le d} q_{ii}\xi_i^2+2\sum_{\begin{subarray}{c}
          1\le i,j\le d\\[0.4ex]i<j \end{subarray}} q_{ij}\xi_i\xi_j,\qquad \xi:=\left(\xi_1,\dots,
            \xi_d\right)\in\R^d.
\end{equation}
The following trigonometric polynomial
has the required Taylor expansion about zero, see~\eqref{m_1SimP},
\begin{equation}\label{m_1}
  {G}(\xi_1,\dots,\xi_d) := 4\sum_{1\le i\le d} q_{ii}\sin^2\frac{\xi_i}2
    +2\sum_{\begin{subarray}{c}
          1\le i,j\le d\\[0.4ex]i<j \end{subarray}}q_{ij}\sin\xi_i\sin\xi_j.
\end{equation}
Thus, using~\eqref{m_1}, the mask $m_0$ defined by~\eqref{m_0Definition} is a trigonometric polynomial.

In the next lemma, we investigate zeros of trigonometric polynomial~\eqref{m_1}.
\begin{lemma}
  For any quadratic positive definite form~\eqref{AnotherQuadraticForm}, trigonometric
polynomial~\eqref{m_1}
is not negative on $\R^d$ and vanishes only at the points $2\pi k$, $k\in\Z^d$.
\end{lemma}
\begin{proof}

Rewrite formula~\eqref{m_1} as follows
$$
  {G}(\xi_1,\dots,\xi_d) = 4\sum_{1\le i\le d} q_{ii}\sin^2\frac{\xi_i}2
     + 8\sum_{\begin{subarray}{c}
          1\le i,j\le d\\[0.4ex]i<j \end{subarray}}q_{ij}
              \sin\frac{\xi_i}2\sin\frac{\xi_j}2\cos\frac{\xi_i}2\cos\frac{\xi_j}2.
$$
Since the quadratic form $W(\xi)$
is positive definite; we have
\begin{multline*}
  \sum_{1\le i\le d} q_{ii}\sin^2\frac{\xi_i}2
     + 2\sum_{\begin{subarray}{c}
          1\le i,j\le d\\[0.4ex]i<j \end{subarray}}q_{ij}
              \sin\frac{\xi_i}2\sin\frac{\xi_j}2 \\ \equiv
                  W\left(\sin\frac{\xi_1}2,\dots,\sin\frac{\xi_d}2\right)\ge0,\qquad
                     \forall\xi\in\R^d,
\end{multline*}
and the trigonometric polynomial $W\left(\sin\dfrac{\xi_1}2,\dots,\sin\dfrac{\xi_d}2\right)$
vanishes iff $\sin\dfrac{\xi_j}2=0,\ j=1,\dots,d$.
Since $0\le\cos\xi_j/2\le1$, $\xi_j\in[-\pi,\pi]$, $j=1,\dots,d$, and
$G(\xi)$ is $2\pi$-periodic; it follows that ${G}(\xi)\ge0$ for all $\xi\in\R^d$
and ${G}(\xi)$ vanishes only at the points $2\pi k$, $k\in\Z^d$.
\end{proof}

Thus, for trigonometric polynomial~\eqref{m_1},
conditions~\eqref{NonZeroConditionsOnG} are satisfied automatically.

\begin{remark}
Any elliptic scaling function corresponding
to trigonometric polynomial~\eqref{m_1} is {\em compactly supported}, see~\cite{CDM}.
\end{remark}

Here we shall not discus other
properties of the elliptic scaling functions and we refer the reader to~\cite{ZakhConstrApprox}.


\subsection{Higher order scaling functions}

Taking an elliptic scaling function $\phi$ (of the first order) defined above,
the elliptic scaling function of order $m=2,3,\dots$, denoted by $\phi^m$, is given (in the Fourier domain)
as follows
\begin{equation}\label{ScalingFunctionOfArbitraryOrder}
  \hat\phi^m(\xi) := \left(\hat\phi(\xi)\right)^m,
\end{equation}
and the corresponding mask $m_0^m(\xi)$ is given as: $m_0^m(\xi):=\left(m_0(\xi)\right)^m$,
where $m_0$ is the mask corresponding to $\phi$.

\section{Reproducing of polynomials}\label{Sect:ReproductionOfPolynomials}

\subsection{Main theorem}\label{Subsection:MainTheorem}

\begin{lemma}\label{HatPhiZeros}
Let $A$ be an isotropic dilation matrix. Let the matrix $A^{-T}$
be factored by~\eqref{MatrixBDecomposition} and a mask $m_0$ be given by~\eqref{m_0Definition},
where the trigonometric polynomial $G$ is given by~\eqref{m_1}
and the quadratic form $W$
is given by~\eqref{QuadraticForm}.
Suppose that $\phi$ is the (elliptic) scaling function corresponding to the dilation matrix $A$ and the mask $m_0$; then
$\hat\phi(2\pi n)=0$ for all $n\in\Z^d\setminus\{0\}$. Moreover, the Taylor series
of
$\hat\phi$
about the points $2\pi n$, $n\in\Z^d\setminus\{0\}$,  are of the form
\begin{equation*}
  \hat\phi(\xi-2\pi n) \propto W(\xi)
    + R_4(\xi).
\end{equation*}
\end{lemma}
\begin{proof}
Since the mask $m_0$ is $2\pi$-periodic; therefore, $m_0$ vanishes at the points
$2\pi s+2\pi n$, $s\in\cS(A^T)\setminus\{0\}$, $n\in\Z^d$, and the function $m_0((A^T)^{-j}\cdot)$,
$j=1,2,\dots$, vanishes
at the points $2\pi (A^T)^{j}(s+n)$, $s\in\cS(A^T)\setminus\{0\}$, $n\in\Z^d$.
By Theorem~\ref{Theorem_UnionSets}, we have $\set{2\pi (A^T)^{j}(s+n)}{s\in\cS(A^T)\setminus\{0\},n\in\Z^d,j\in\N}
=2\pi\Z^d\setminus\{0\}$, thus $\hat\phi$ vanishes at $2\pi\Z^d\setminus\{0\}$.

Moreover, since the sets $\set{2\pi (A^T)^{j}(s+n)}{n\in\Z^d}$,
$s\in\cS(A^T)\setminus\{0\}$, $j\in\N$,
do not intersect for different $j$ and $s$;
the zeros of $m_0((A^T)^{-j}\cdot)$, $j\in\N$, do not superimpose and have the multiplicity
coincided with the multiplicity of the zeros of $m_0(\cdot)$.
Hence, by invariance property~\eqref{QuadrFormInvariance},
the terms of the second degree of the Taylor series for $\hat\phi$
at the points $2\pi\Z^d\setminus\{0\}$
are proportional to the second degree terms of
the Maclaurin series of trigonometric polynomial~\eqref{m_1}.
And the terms of the third degree are zero.
\end{proof}

Now we can state and prove the main theorem of this section.

\begin{theorem}\label{ScalingFunction:MainTheorem}
Let $A$ be an isotropic dilation matrix and let
a mask $m_0$ be given by~\eqref{m_0Definition},
\eqref{m_1},
where the quadratic form $W$
is given by~\eqref{QuadraticForm}.
Suppose the $m$th-order, $m=1,2,\dots$, elliptic scaling function $\phi^m$,
defined by~\eqref{ScalingFunctionOfArbitraryOrder},
corresponds to the dilation matrix $A$ and the mask $m_0$;
then any algebraic polynomial $P$
reproduced by integer
shifts of the scaling function $\phi^m$ belongs to $\ker W(-iD)^m$.
Moreover, the polynomial space contained in the span of integer shifts of $\phi^m$ is affinely invariant.
\end{theorem}
\begin{proof}
By Lemma~\ref{HatPhiZeros}, it follows that $\ker\cc\Delta_{2m+1}\hat\phi^m=\ker\cc\bD_{2m+1}W(0)^m$;
and the matrix $\cc\bD_{2m+1}W(0)^m$ is of the form
\begin{equation}\label{MatrixFormDW}
  \cc\bD_{2m+1}W(0)^m =
  \begin{bmatrix}
    0 & \ldots & 0 & \bD^0_{2m}W(0)^m& 0            \\
    0 & \ldots & 0 & 0       & \bD^1_{2m+1}W(0)^m  \\
    0 & \ldots & 0 & 0       & 0            \\
    \vdots & \ddots & \vdots & \vdots  & \vdots      \\
    0 & \ldots & 0 & 0       & 0            \\
  \end{bmatrix}.
\end{equation}

Suppose $V:=\ker \cc\bD_{2m+1}W(0)^m$ and take a vector $v\in V$; then, by Theorem~\ref{Theorem:PolynomialFromKernel},
the algebraic polynomial $\left[\cc\cP_{2m+1}\right]v$ belongs to $\ker W(-iD)^m$.
Since $v$ is an arbitrary vector from $V$; it follows that any polynomial
contained in $\Span\left\{\phi^m(\cdot-k),k\in\Z^d\right\}$ belongs to $\ker W(-iD)^m$.

Since
the nonzero submatrices
are situated on a
diagonal of
matrix~\eqref{MatrixFormDW}; by Theorem~\ref{Theorem_Scale-Invariance}, we see that the polynomial space
contained in the span of integer shifts of $\phi^m$
is affinely invariant. This concludes the proof.
\end{proof}

\subsubsection{Examples}

\paragraph{Quincunx dilation matrix.}
Here we consider a quincunx dilation matrix
\begin{equation}\label{Pi/4_DilationMatrix}
  A:=\begin{bmatrix}
    1 & 1 \\
    1 & -1 \\
  \end{bmatrix}.
\end{equation}
The matrix $A$ is isotropic;
then the matrix $A^{-T}$
can be presented of the form~\eqref{MatrixBDecomposition},
where $U:=\begin{bmatrix}
\frac{1}{\sqrt2} & \frac{1}{\sqrt2}\\ \frac{1}{\sqrt2} & -\frac{1}{\sqrt2}\end{bmatrix}$
and $Q$ is the identity matrix.
Hence the quadratic form $W(\xi)$ is $|\xi|^2$, $\xi\in\R^2$;
and the corresponding differential operator is
the Laplace operator $\Delta:=\partial_{xx}+\partial_{yy}$, $(x,y)\in\R^2$.
Trigonometric polynomial~\eqref{m_1} is given as $G(\xi_1,\xi_2)
:=4\left(\sin^2(\xi_1/2)+\sin^2(\xi_2/2)\right)$;
and, for matrix~\eqref{Pi/4_DilationMatrix}, $\cS(A^T):=\left\{(0,0), (1/2,1/2)\right\}$,
consequently the mask is of the form
\begin{equation}\label{Pi/4_Mask}
  m_{0}(\xi_1,\xi_2):=\frac12+\frac14\cos\xi_1+\frac14\cos\xi_2.
\end{equation}
Then it is easily seen that the elliptic scaling function $\phi$ corresponding to dilation
matrix~\eqref{Pi/4_DilationMatrix} and mask~\eqref{Pi/4_Mask} satisfies
Theorem~\ref{ScalingFunction:MainTheorem}.
So the polynomials reproduced by the integer shifts of the scaling function
belong to the null-space of the Laplace operator. Indeed, in accordance with Lemma~\ref{HatPhiZeros},
the order of the Strang--Fix conditions is $3$ and there are two nonzero submatrices of the matrix $\cc\bD_3W(0)$:
$\bD^0_2W(0)=-\begin{bmatrix}
2 & 0 & 2\end{bmatrix}$ and $\bD^1_3W(0)=-\begin{bmatrix}
6 & 0 & 2 & 0\\0 & 2 & 0 & 6\end{bmatrix}$. Consequently,
\begin{multline*}
  \cV:=\Pi\cap\Span\set{\phi(\cdot-k)}{k\in\Z^2}=\Pi_{\le1}\oplus\Span\left\{x^2-y^2,xy\right\}\\ \oplus
    \Span\left\{x^3-3xy^2,y^3-3x^2y\right\}.
\end{multline*}
Thus we have $\cV\subset\ker\Delta$ and the space $\cV$ is affinely invariant.

\begin{remark}
The elliptic scaling function corresponding to matrix~\eqref{Pi/4_DilationMatrix} and
mask~\eqref{Pi/4_Mask}
has been considered in detail in the paper~\cite{ZakhConstrApprox}. Note also that,
in the context of the construction of biorthogonal masks,
mask~\eqref{Pi/4_Mask} has been proposed in the book~\cite{NPS}.
\end{remark}

\paragraph{Another dilation matrix.}
Now we consider the following dilation matrix
\begin{equation}\label{DilationMatricesTr1}
    A
      :=\begin{bmatrix}
          1 & -2 \\
          1 & 0 \\
        \end{bmatrix}.
\end{equation}
Matrix~\eqref{DilationMatricesTr1} is isotropic;
and $\cS(A^T)=\left\{(0,0), (1/2,1/2)\right\}$.
The matrix $A^{-T}$
can be presented of the form~\eqref{MatrixBDecomposition},
where the orthogonal matrix is of the form
\begin{equation}\label{RotationMatrixTr1}
  U
  :=\begin{bmatrix}
            \frac{1}{2\sqrt2} & -\frac{\sqrt7}{2\sqrt2} \\
            \frac{\sqrt7}{2\sqrt2} & \frac{1}{2\sqrt2} \\
   \end{bmatrix}.
\end{equation}
and the square of the corresponding similarity transformation matrix is
$$
  Q^{2}
    :=\begin{bmatrix}
        2 & \frac12 \\
        \frac12 & 1 \\
      \end{bmatrix}.
$$
Thus we have $W(\xi_1,\xi_2):=2\xi_1^2+\xi_1\xi_2+\xi_2^2$ and
$G(\xi_1,\xi_2):=8\sin^2(\xi_1/2)+4\sin^2(\xi_2/2)+\sin\xi_1\sin\xi_2$. Therefore
the mask is
\begin{equation}\label{MaskTr1}
  m_{0}(\xi_1,\xi_2):=\frac12+\frac{1}{3}\cos\xi_1+\frac16\cos\xi_2+\frac1{12}\sin\xi_1\sin\xi_2.
\end{equation}

Similarly to the previous dilation matrix we can define a polynomial space contained
in the span of integer shifts of the scaling function corresponding to dilation
matrix~\eqref{DilationMatricesTr1} and mask~\eqref{MaskTr1}:
\begin{multline}\label{PolynomialSpaceSecondExample}
  \Pi_{\le1}\oplus\Span\left\{x^2-4xy,y^2-2xy\right\}\\
    \oplus\Span\left\{x^3+6x^2y-12xy^2,y^3-3x^2y+3xy^2\right\}.
\end{multline}
It is easy to see that polynomial space~\eqref{PolynomialSpaceSecondExample} belongs to the null-space of the operator
$2\partial_{xx}+\partial_{xy}+\partial_{yy}$
and is affinely invariant.

\begin{remark}
Note that orthogonal matrix~\eqref{RotationMatrixTr1} realizes the rotation by an angle
such that the angle is incommensurate with $\pi$. The rotation properties of isotropic dilation matrices
will be the object of another paper.
\end{remark}

\paragraph{Diagonal dilation matrix.}
Finally we consider a diagonal dilation matrix
\begin{equation}\label{DiagonalDilationMatrix}
  A:=\begin{bmatrix}
    2 & 0 \\
    0 & 2 \\
  \end{bmatrix}.
\end{equation}
Formally, the orthogonal matrix for this matrix is the identity matrix and
the similarity transformation matrix is also the identity matrix.
It is surprising that,
for matrix~\eqref{DiagonalDilationMatrix},
it is also possible to construct an elliptic scaling function.
Indeed,
the quadratic form corresponding to matrix~\eqref{DiagonalDilationMatrix} is $W(\xi_1,\xi_2):=\xi_1^2+\xi_2^2$,
$(\xi_1,\xi_2)\in\R^2$, the set $\cS(A^T)\setminus\{(0,0)\}:=\left\{(1/2,1/2),(0,1/2),(1/2,0)\right\}$,
and the mask is of the form
\begin{multline}\label{MaskForDiagonalDilationMatrix}
   m_0(\xi_1,\xi_2)\\:=\frac1{16}\left(2+\cos\xi_1+\cos\xi_2\right)\left(2+\cos\xi_1-\cos\xi_2\right)
        \left(2-\cos\xi_1+\cos\xi_2\right).
\end{multline}
Consequently the elliptic scaling function corresponding to mask~\eqref{MaskForDiagonalDilationMatrix} and
matrix~\eqref{DiagonalDilationMatrix} must reproduce polynomials from the null-space of the Laplace operator.

Note that any homogeneous polynomial is invariant (within a constant factor) under the coordinate transformation
by matrix~\eqref{DiagonalDilationMatrix}. This invariance property will be used under the construction of
the scaling functions that reproduce not scale-invariant polynomial spaces, see the second example in
Subsection~\ref{SubsubSection:NotScaleInvariantSpaces}.

\subsection{Higher degree polynomials}\label{Subsection:HigherDegreePolynomials}

It is possible to generalize the elliptic scaling functions in such a way that the scaling functions will
reproduce polynomials of a higher degree.
Namely we can state the following theorem.

\begin{theorem}\label{Theorem:HigherOrderPolynomialsReproduceationByScFunction}
Let $A$ be an isotropic dilation matrix and let
a mask $m_0$ be given by~\eqref{m_0Definition},
where the quadratic form $W$
is given by~\eqref{QuadraticForm} and
the Maclaurin series of the trigonometric polynomial $G$
is of the form
\begin{equation}\label{TrigPolynomialHigherOrderPolynomialReproduceation}
  {G}(\xi) := W(\xi) + R_r(\xi),
\quad r\ge 6.
\end{equation}
Suppose the elliptic scaling function $\phi$ corresponds to the dilation matrix $A$ and the mask $m_0$;
then the scaling function $\phi^m$, $m=1,2,\dots$, reproduces polynomials
up to degree $2m+r-3$ and the polynomials belong to $\ker W(-iD)^m$. Moreover,
the polynomial space contained in the span of integer shifts of $\phi^m$
is affinely
invariant.
\end{theorem}
The proof is left to the reader.

\begin{remark}\label{Remark:TrigonometricPolynomialsConstruction}
Here we shall not present any explicit method
to construct trigonometric
polynomial~\eqref{TrigPolynomialHigherOrderPolynomialReproduceation}.
Note only that, subtracting appropriate trigonometric polynomials
$\sum\limits_{k=4,6,\dots,r-2}P_k\left(\sin\xi_1,\dots,\sin\xi_d\right)$,
where $P_k\in\Pi_k$, from trigonometric polynomial~\eqref{m_1};
we always can obtain the required trigonometric polynomial.
\end{remark}

\subsubsection{Examples}

Here we consider generalizations of the scaling function corresponding to quincunx dilation
matrix~\eqref{Pi/4_DilationMatrix}.

First we present the following trigonometric polynomial corresponding to the quadratic form
$W(\xi_1,\xi_2):=\xi_1^2+\xi_2^2$:
$$
  G(\xi_1,\xi_2):=4\left(\sin^2(\xi_1/2)+\sin^2(\xi_2/2)+\frac13\sin^4(\xi_1/2)+\frac13\sin^4(\xi_2/2)\right).
$$
It is easy to see that the Taylor series of $G$ at zero point is of the form
$G(\xi_1,\xi_2):=W(\xi_1,\xi_2)+ R_6(\xi_1,\xi_2)$;
and the corresponding mask
is
\begin{equation}\label{HigherDegreePolynomialsMask}
  m_0(\xi_1,\xi_1):=\frac{1}{32}\left(15+8\cos\xi_1+8\cos\xi_2+\frac1{2}\cos2\xi_1+
    \frac1{2}\cos2\xi_2\right).
\end{equation}
Now, by Theorem~\ref{Theorem:HigherOrderPolynomialsReproduceationByScFunction}, we see that the
span of integer shifts of the scaling function corresponding to matrix~\eqref{Pi/4_DilationMatrix} and
mask~\eqref{HigherDegreePolynomialsMask} contains polynomials up to degree $5$ and the polynomials
belong to the null-space of the Laplace operator.

Similarly, the following mask
\begin{multline}\label{HigherDegreePolynomialsMask1}
  m_0(\xi_1,\xi_1):=\frac{1}{544}\biggl( 245+135\cos\xi_1+135\cos\xi_2\\
    +\frac{27}{2}\cos2\xi_1+\frac{27}{2}\cos2\xi_2
       +\cos3\xi_1+\cos3\xi_2\biggr)
\end{multline}
gives the scaling function that reproduces polynomials up to degree $7$ and the polynomials belong
to the null-space of the Laplace operator.

The polynomial spaces reproduced by
the scaling functions corresponding to
masks~\eqref{HigherDegreePolynomialsMask}, \eqref{HigherDegreePolynomialsMask1}
can be obtained by, for example,
Theorem~\ref{Theorem:PolynomialFromKernel}.

\begin{remark}
It is interesting to note that the polynomial spaces reproduced by the elliptic scaling functions (and contained
in the null-spaces of the corresponding differential operators) are invariant under the coordinate transformation by
the isotropic dilation matrices (that define the corresponding elliptic scaling functions and operators).
This will be discussed elsewhere.
\end{remark}

\subsection{Not affinely invariant case}\label{Subsect:NotAffinelyInvariantCase}

\subsubsection{Not scale-invariant differential operators}

As it has been noted, all the elliptic scaling functions discussed above cannot reproduce
not affinely invariant polynomial spaces.
However we can conjecture that, in the case of not scale-invariant differential operators, the corresponding
elliptic scaling functions reproduce not scale-invariant polynomial
spaces.

Consider a sum of homogeneous elliptic differential
operators of the form
$$
  \te W(-iD):=\sum_{k=k_1}^{k_2} C_k W(-iD)^k,\quad C_k\in\R,\ k_2>k_1\ge1,
$$
where $W$ is quadratic form~\eqref{QuadraticForm}.
Obviously, the operator $\te W(-iD)$ is not homogeneous and consequently not scale-invariant;
and the symbol of the operator
\begin{equation}\label{NoScaleInvariantPolynomial}
    \te W(\xi):= \sum_{k=k_1}^{k_2} C_k W(\xi)^k,\quad \xi\in\R^d,
\end{equation}
is also not scale-invariant.
Thus the corresponding elliptic scaling function must be nonstationary.
In the context of nonstationary masks construction,
we are interested in the transformation of polynomial~\eqref{NoScaleInvariantPolynomial}
by the matrix $A^T$ corresponding to the quadratic form $W$:
$$
  \te W\left(\left(A^{-T}\right)^j\xi\right)
    = \sum_{k=k_1}^{k_2} \frac{C_k}{q^{2kj/d}} W(\xi)^k
    = \frac{1}{q^{2k_1j/d}}\sum_{k=k_1}^{k_2} \frac{C_k}{q^{2(k-k_1)j/d}} W(\xi)^k,
$$
where $q=|\det A|$.
Suppose the Maclaurin series
of a trigonometric polynomial $\ssc0G$ is of the form
\begin{equation}
  \ssc0G(\xi):=\te W(\xi)
    + \ssc 0R_{>2k_2}(\xi).
\end{equation}
Since polynomial~\eqref{NoScaleInvariantPolynomial} is not scale-invariant; for any scale number $j\in\Z$,
we define a trigonometric polynomial $\ssc jG$ such that its Maclaurin series is
\begin{equation}\label{NonstationaryTrigPolynomialExpansion}
  \ssc jG(\xi):=\sum_{k=k_1}^{k_2} C_k q^{2(k-k_1)j/d} W(\xi)^k+\ssc jR_{>2k_2}(\xi).
\end{equation}
Now, for any scale $j$, the terms in the Maclaurin series of
$\ssc{j}G\left(\left(A^{-T}\right)^j\cdot\right)$, the degree of which is
less than or equal to $2k_2$,
are proportional to the similar terms of Maclaurin's series of $\ssc0G$.
Indeed, we have
\begin{multline*}
  \ssc{j\,}G\left(\left(A^{-T}\right)^j\xi\right)
      =\sum_{k=k_1}^{k_2} \frac{C_k q^{2(k-k_1)j/d}}{q^{2kj/d}} W(\xi)^k+\ssc j{R^*}_{>2k_2}(\xi) \\
     =\dfrac{1}{q^{2k_1j/d}}\sum_{k=k_1}^{k_2} C_k W(\xi)^k+\ssc j{R^*}_{>2k_2}(\xi)
       =\dfrac{1}{q^{2k_1j/d}}\te W(\xi)+\ssc j{R^*}_{>2k_2}(\xi),
\end{multline*}
where $\ssc j{R^*}_{>2k_2}(\xi):=\ssc j{R}_{>2k_2}\left(\left(A^{-T}\right)^j\xi\right)$.
\begin{remark}

Trigonometric polynomials~\eqref{NonstationaryTrigPolynomialExpansion} can be obtained
similarly to the trigonometric polynomials from the previous subsection,
see Remark~\ref{Remark:TrigonometricPolynomialsConstruction}.
\end{remark}

Unfortunately we have the following statement.
\begin{statement}
The polynomial spaces reproduced by the nonstationary elliptic scaling functions corresponding
to trigonometric polynomials~\eqref{NonstationaryTrigPolynomialExpansion} are scale-invariant as before.
\end{statement}

\begin{lemma}\label{MatrixDlP(0)}
Let $L\ge0$, $m\ge2$. Let a homogeneous polynomial $P$ belong to $\Pi_k$, $k\ge1$.
Suppose $L-mk\ge0$;
then $\ker\bD_L^{L-k}P(0)\subseteq\ker\bD_L^{L-mk}P(0)^m$.
\end{lemma}
\begin{proof}[Proof of the lemma]
Since the matrices $\cc\bD_LP(0)$ and $\cc\bD_LP(0)^m$ are block $k$- and $mk$-diagonal matrices, respectively;
it follows that without loss of generality we can consider only the matrices $\bD_LP(0)$ and $\bD_LP(0)^m$.
Obviously, the nonzero submatrices contained in the previous matrices
are
$\bD_L^{L-k}P(0)$ and $\bD_L^{L-mk}P(0)^m$, respectively.
Suppose $v\in\ker\bD_L^{L-k}P(0)$; then, by Theorem~\ref{Theorem:PolynomialFromKernel}, we have
$\left[\cP_L\right]v\in\ker P(-iD)$. Since $\ker P(-iD)\subseteq \ker P(-iD)^m$,
we have $v\in\ker\bD_L^{L-mk}P(0)^m$.
\end{proof}

\begin{proof}[Proof of the statement]
Suppose that $L$ is the order of the Strang--Fix conditions; then
$L\ge 2k_2$ and the degree of the higher order terms is greater than $L$.
Similarly to Theorem~\ref{ScalingFunction:MainTheorem}, we must consider a matrix $\cc\bD_L\te W(0)$.
The matrix $\cc\bD_L\te W(0)$ contains the $k_2-k_1+1$ block diagonals of nonzero submatrices.
Without loss of generality consider the matrix $\bD_L\te W(0)$, which is the rightmost column-submatrix
of the matrix $\cc\bD_L\te W(0)$.
By Lemma~\ref{MatrixDlP(0)}, and since $\ker\bD_L^{L-2k_1}\te W(0)\subseteq\ker\bD_L^{L-2k_1-1}\te W(0)
\subseteq\dots\subseteq\ker\bD_L^{L-2k_2}\te W(0)$;
we have
$$
  \ker\bD_L\te W(0)=\bigcap_{k=k_1}^{k_2}
     \ker\bD_{L}^{L-2k}\te W(0)=\ker\bD_{L}^{L-2k_1}\te W(0).
$$
The analogous relations are valid for all the
matrices $\bD_l\te W(0)$, $l=2k_1,\dots,L$.
Thus we have the same situation as
in Subsections~\ref{Subsection:MainTheorem},\,\ref{Subsection:HigherDegreePolynomials}.
Hence the polynomial subspace corresponding to the matrix $\cc\bD_L\te W(0)$
is affinely invariant, and the corresponding nonstationary elliptic scaling function can reproduce
an affinely invariant polynomial space only.
\end{proof}

Also we have the following corollary.
\begin{corollary}
The polynomial subspace of the null-space of a differential operator
$\sum_{k=k_1}^{k_2} C_k W(-iD)^k$, $C_k\in\R$, where $W$ is a homogeneous polynomial,
is affinely invariant and
coincides with the polynomial subspace of the null-space of the lowest order operator $W(-iD)^{k_1}$.
Note that the lowest degree $k_1$ must be nonzero.
\end{corollary}

\subsubsection{Not scale-invariant spaces after all!}

However we can offer an approach to construct nonstationary elliptic scaling functions
that reproduce not scale-invariant (only shift-invariant) polynomial spaces.

\begin{theorem}\label{Theorem:NotScale-InvariantSpaces}
Let $m\in\N$.
Let a homogeneous polynomial $X$ belong to $\Pi_k$,
$1\le k\le 2m-1$.
Let $A$ be an isotropic dilation matrix, and let
the quadratic form $W$ be given by~\eqref{QuadraticForm}.
Let masks $\ssc{j}{m_0}$, $j\in\Z$, be given by~\eqref{m_0Definition}, where
the Maclaurin series of trigonometric polynomials $\ssc jG\left(\left(A^{-T}\right)^j\cdot\right)$
are of the form
\begin{equation}\label{TrigPolynomialForNotScaleInvariantCase}
  \ssc jG\left(\left(A^{-T}\right)^j\xi\right) := C_j\Bigl(X(\xi) + W(\xi)^m\Bigr) + \ssc jR_{>2m}
     \left(\left(A^{-T}\right)^j\xi\right),
\end{equation}
where $C_j$ is a constant factor that depends on scale $j$.
Moreover, suppose that
\begin{equation}\label{SetsMinusConditions}
  \Pi\cap\Bigl(\ker X(-iD)\setminus\ker W(-iD)^m\Bigr)\ne\emptyset;
\end{equation}
then the nonstationary elliptic scaling function $\phi$ corresponding to the dilation matrix $A$
and the mask $m_0$
can reproduce not scale-invariant
(only shift-invariant) polynomial spaces.
(Here and in the sequel,
by $\phi$, $m_0$, and $G$ we denote the scaling function, mask, and trigonometric
polynomial, respectively, of zero number ($=$ zero scale).)
\end{theorem}

\begin{lemma}\label{Lemma:Setminus}
Under the conditions of Theorem~\ref{Theorem:NotScale-InvariantSpaces}, we see that
the following conditions are equivalent:
\begin{itemize}
\item[(i)] $\Pi\cap\Bigl(\ker X(-iD)\setminus\ker W(-iD)^m\Bigr)\ne\emptyset$;
\item[(ii)] $\ker\bD_L^{L-k}X(0)\setminus\ker\bD_L^{L-2m}W(0)^m\ne\emptyset$;
\end{itemize}
where $L$ is the order of the Strang--Fix conditions.
\end{lemma}
We omit the proof of the lemma and note only that the proof is based on Theorem~\ref{Theorem:PolynomialFromKernel}
and is similar to the proof of Lemma~\ref{MatrixDlP(0)}.

\begin{proof}[Proof of the theorem]
Let $L$ be the order of the Strang--Fix conditions and let
\begin{equation}\label{Lmk}
  L-2m+k<2m\quad\Longleftrightarrow\quad L<4m-k.
\end{equation}
By Lemma~\ref{HatPhiZeros}, it follows that $\ker\cc\Delta_{L}\hat\phi=\ker\cc\bD_{L}G(0)$.
The matrix $\cc\bD_{L}G(0)$ has $k$- and $2m$-diagonals of nonzero submatrices: $\bD^j_{j+k}X(0)$,
$j=0,\dots,L-k$, and $\bD^{j}_{j+2m}W(0)^m$, $j=0,\dots,L-2m$, respectively.
Since the matrix $\cc\bD_L G(0)$ is upper triangular and singular,
there exists a nonzero linear space
$V:=\ker\cc\bD_L G(0)$.
Consider the rightmost column-submatrix:
$$
  \bD_L G(0)
  =\begin{bmatrix}
     0 & \cdots & \bD^{L-2m}_{L}W(0)^m & \cdots & 0 & \cdots & \bD^{L-k}_{L}X(0) & \cdots & 0 \\
   \end{bmatrix}^T.
$$

First we prove that the subspace $V^L\subseteq V$,
see~\eqref{RDecomposition}, cannot be the zero space.
Assume the converse: $V^L=\{(0,\dots,0)\}$.
Since the submatrices on the block $k$-diagonal of the matrix $\cc\bD_L G(0)$ are full-rank matrices,
we have $\dim\ker\cc\bD_L G(0) = \cc d(L)- \cc d(L-k)$.
Thus,
\begin{align*}
  \dim&\ker\cc\bD_L G(0)\\ &= \dim\ker\begin{bmatrix}\bD_0G(0)&\bD_1G(0)&\cdots&\bD_{L-1}G(0)&0\end{bmatrix}\\
     &= \dim\ker\cc\bD_{L-1} G(0) = \cc d(L-1)- \cc d(L-k-1)<\cc d(L)- \cc d(L-k),\\ &
     \qquad\qquad\qquad\qquad\qquad\qquad\qquad\qquad\qquad\qquad\qquad\quad 1\le k<2m\le L.
\end{align*}
This contradiction proves that $V^L\ne\{(0,\dots,0)\}$.

By condition~\eqref{SetsMinusConditions} and Lemma~\ref{Lemma:Setminus},
there exists
a vector $v^L\in\R^{d(L)}$ such that
$v^L\in\ker\bD_L^{L-k}X(0)$ and $v^L\not\in\ker\bD_L^{L-2m}W(0)^m$. Hence we have $v^L\not\in\ker\bD_L G(0)$.

Consider the block $(L-2m)$-row of the matrix $\cc\bD_{L}{G(0)}$:
$$
  \left[\cc\bD_{L}{G(0)}\right]_{L-2m} :=
  \begin{bmatrix}
     0 & \cdots & \bD^{L-2m}_{L-2m+k}X(0) & \cdots & 0 & \cdots & \bD_L^{L-2m}W(0)^m\\
  \end{bmatrix},
$$
where $\bD^{L-2m}_{L-2m+k}X(0)$ is a submatrix situated at the intersection of $k$-diagonal and
$(L-2m)$-row.
Since the matrix $\bD^{L-2m}_{L-2m+k}X(0)$ has the full-rank, it follows that
there exists a non-zero subvector $v^{L-2m+k}\in\R^{d(L-2m+k)}$
such that the following vector
$$
  v:=\left(0,\dots,0,v^{L-2m+k},0,\dots,0,v^L\right)\in\R^{\cc d(L)}
$$
belongs to $\ker\left(\left[\cc\bD_{L}{G(0)}\right]_{L-2m}\right)$. By~\eqref{Lmk}, we see
that $L-2m+k$ block column-matrix contains only one nonzero submatrix $\bD^{L-2m}_{L-2m+k}X(0)$,
and
since $v^L\in\ker\bD_L^{L-k}X(0)$,
it follows that the vector $v$ belongs to the null-space of $\cc\bD_LG(0)$.
Thus, by Theorem~\ref{Theorem_Scale-Invariance}, we see that
the polynomial space corresponding to
$\ker\cc\bD_LG(0)$ is not scale-invariant.

(If inequality~\eqref{Lmk} is not valid; then, for some $l\in\N$, $2m\le l<L$, such that $l<4m-k$,
me can consider a submatrix $\cc\bD_lG(0)$ of the matrix $\cc\bD_LG(0)$.)
\end{proof}

Since the polynomial $X + W^m$, from expansions~\eqref{TrigPolynomialForNotScaleInvariantCase},
is not scale-invariant, we can use the approach discussed in the previous subsection,
see~\eqref{NonstationaryTrigPolynomialExpansion}.
However we have another complication.
Since the polynomial $X$ is not invariant under
coordinate transformation by the isotropic dilation matrix and
the dilation matrix (actually the corresponding orthogonal matrix) forms a cyclic group
of some order $n$ (infinite cyclic groups also included);
it follows that we must construct $n$ appropriate trigonometric polynomials
$\ssc{j}{G}$ such that their Maclaurin series begin with $X$.
The trigonometric polynomials $\ssc{j}{G}$ can be obtained
similarly to the polynomials from the previous subsections, see
Remark~\ref{Remark:TrigonometricPolynomialsConstruction}.

\begin{remark}
Of course, we considered in Theorem~\ref{Theorem:NotScale-InvariantSpaces}
the simplest case of the polynomials that supply not scale-invariant
polynomial spaces in the spans of integer shifts of the corresponding elliptic scaling functions. In
particular, the polynomial $X$ can be not necessarily homogeneous. Also
the degree of $X$ can be greater than the degree of the polynomial $W^m$.
Then condition~\eqref{SetsMinusConditions} must be rewritten as
$$
  \Pi\cap\Bigl(\ker W(-iD)^m\setminus\ker X(-iD)\Bigr)\ne\emptyset.
$$
\end{remark}

\subsubsection{Examples}\label{SubsubSection:NotScaleInvariantSpaces}

\paragraph{Quincunx dilation matrix.}
For the matrix $A$ given by \eqref{Pi/4_DilationMatrix},
we have
$$
  A^j
  =\left\{\begin{array}{ll}
            2^{j/2}I
               &  \mbox{ if $j$ is even};\\[2ex]
            2^{(j+1)/2}A & \mbox{ if $j$ is odd},\\
          \end{array}\right.
$$
where $I$ is the $2\times2$ identity matrix.
Thus the order of the cyclic group corresponding to matrix~\eqref{Pi/4_DilationMatrix} is $2$ and
the group consists of the elements: $\left\{\frac{1}{\sqrt2}A,
I\right\}$.
As it was considered above the homogeneous polynomial $W$ corresponding to isotropic matrix~\eqref{Pi/4_DilationMatrix}
is the quadratic form $W(\xi_1,\xi_2):=\xi_1^2+\xi_2^2$. Take another polynomial as $X(\xi_1,\xi_2):=2i\xi_1$;
then the polynomial $\xi_2^2$ belongs to $\ker X(-iD)$ and does not belong to $\ker W(-iD)$.
Now the trigonometric polynomials
\begin{multline*}
  \ssc jG(\xi_1,\xi_2)\\
  := \left\{
       \begin{aligned}
         &4\left(\sin^2(\xi_1/2)+\sin^2(\xi_2/2)\right)+\frac{2i}{2^{j/2}}\sin\xi_1  && \mbox{ if $j$ is even};\\
         &4\left(\sin^2(\xi_1/2)+\sin^2(\xi_2/2)\right)+\frac{2i}{2^{(j+1)/2}}\left(\sin\xi_1+\sin\xi_2\right)
             && \mbox{ if $j$ is odd}
        \end{aligned}
      \right.
\end{multline*}
have the following Maclaurin series of the functions $\ssc jG\left(\left(A^T\right)^{-j}\cdot\right)$:
$$
  \ssc jG\left(\left(A^T\right)^{-j}\begin{bmatrix}
  \xi_1 \\ \xi_2
\end{bmatrix}\right)=2i\xi_1+\xi_1^2+\xi_2^2+\ssc jR_{3}(\xi_1,\xi_2);
$$
and the masks are of the form
\begin{equation}~\label{NonstatMasksNotScaleInvariantCase}
  \ssc j{m_0}(\xi_1,\xi_2)
  := \left\{
       \begin{aligned}
         &m_0(\xi_1,\xi_2)-\frac{i}{4}\frac{1}{2^{j/2}}\sin\xi_1  && \mbox{ if $j$ is even};\\
         &m_0(\xi_1,\xi_2)-\frac{i}{4}\frac{1}{2^{(j+1)/2}}\left(\sin\xi_1+\sin\xi_2\right)
             && \mbox{ if $j$ is odd},
        \end{aligned}
      \right.
\end{equation}
where $m_0(\xi_1,\xi_2)$ is given by~\eqref{Pi/4_Mask}.

Suppose that the scaling function of zero scale, denoted by $\phi$,
corresponds to dilation matrix~\eqref{Pi/4_DilationMatrix}
and masks~\eqref{NonstatMasksNotScaleInvariantCase}; then
the matrix $\cc\Delta_3\hat\phi$ and its null-space are of the form
$$
  \cc\Delta_2\hat\phi
  \propto\begin{bmatrix}
       0 & 2 & 0 & -2 & 0 & -2\\
       0 & 0 & 0 & 4 & 0 & 0 \\
       0 & 0 & 0 & 0 & 2 & 0 \\
       \hdotsfor{6}\\
       0 & 0 & 0 & 0 &0 & 0 \\
     \end{bmatrix},
     \qquad\quad
  \ker\cc\Delta_2\hat\phi =
  \begin{bmatrix}
    0 & 1 & 0 & 0 & 0 & 1 \\
    0 & 0 & 1 & 0 & 0 & 0 \\
    1 & 0 & 0 & 0 & 0 & 0 \\
 \end{bmatrix},
$$
respectively.
Consequently the corresponding polynomial space is
\begin{equation*}
  \Span\left\{1,y,x+y^2\right\};
\end{equation*}
and the space belongs to the null-space of the differential operator
$2\partial_{x}-\partial_{xx}-\partial_{yy}$.

\paragraph{Diagonal dilation matrix.}
Here we consider diagonal matrix~\eqref{DiagonalDilationMatrix} again.
As it has been noted,
any homogeneous algebraic polynomial is invariant under transformations by matrix~\eqref{DiagonalDilationMatrix},
consequently
we must define only one trigonometric polynomial such that its Maclaurin series begins with a polynomial
$X$.

Suppose $X(\xi_1,\xi_2):=i\left(\xi_1^3+\xi_2^3\right)$ and $W(\xi_1,\xi_2):=\xi_1^2+\xi_2^2$; then the
trigonometric polynomials are of the form
\begin{multline}
    \ssc jG(\xi_1,\xi_2):= 4\, 2^{-j}\left(\sin^2(\xi_1/2)+\sin^2(\xi_/2)\right)\\+
       8i\left(\sin^3(\xi_1/2)+\sin^3(\xi_2/2)\right)
          =  2^{-j}W(\xi_1,\xi_2)+X(\xi_1,\xi_2)+\ssc jR_4(\xi_1,\xi_2).
          \label{TrigonometricPolynomialsDiagonalDilationMatrix}
\end{multline}
The  (nonstationary) masks corresponding to trigonometric polynomials~\eqref{TrigonometricPolynomialsDiagonalDilationMatrix}
are obtained by formula~\eqref{m_0Definition}.
Let $\phi$ be the corresponding scaling function (of zero scale)
and, determining the null-space
of the matrix $\cc\Delta_3\hat\phi$, we get that
the scaling function $\phi$ reproduces the following not scale-invariant
polynomial space
$$
  \cV:=\Pi_{\le 1}\oplus\Span\left\{x^2-y^2,xy\right\}
    \oplus\Span\left\{3x^2-x^3+3xy^2,3y^2-y^3+3x^2y\right\}.
$$
Note that $\cV\subset\ker\left(\partial_{xx}+\partial_{yy}+\partial_{xxx}+\partial_{yyy}\right)$.

\section*{Acknowledgments}
Research was partially supported by RFBR grant No.~12-01-00608-a.

\end{document}